\documentclass[reqno]{amsart}

\numberwithin{equation}{section} 
\usepackage{amsmath,amsfonts,amssymb,amsthm,latexsym}
\usepackage{enumerate}
\usepackage{cancel}
\usepackage{mathrsfs}
\usepackage{stackrel}

\usepackage{cases}
\usepackage[square,comma,numbers,sort&compress]{natbib}

\usepackage[colorlinks=true]{hyperref}
\hypersetup{urlcolor=blue, citecolor=red}

\newcounter{mnote}
\theoremstyle{plain}
\newtheorem{theorem}{Theorem}[section]
\newtheorem{proposition}[theorem]{Proposition}
\newtheorem{lemma}[theorem]{Lemma}

\theoremstyle{definition}

\theoremstyle{remark}
\newtheorem{remark}[theorem]{Remark}

\newcommand{\field}[1]{\mathbb{#1}}
\newcommand{\nN}{\field{N}}

\newcommand{\nR}{\field{R}}

\newcommand{\cP}{\mathcal P}

\newcommand{\vphi}{\varphi}

\newcommand{\sand}{\quad\text{and}\quad}

\newcommand{\pd}[2]{\frac{\partial #1}{\partial #2}}
\newcommand{\od}[2]{\frac{d #1}{d #2}}

\newcommand{\abs}[1]{\left\lvert#1\right\rvert}
\newcommand{\norm}[1]{\left\lVert#1\right\rVert}
\newcommand{\set}[1]{\left\{#1\right\}}

\newcommand{\LpP}[1]{\text{$L^{#1}$}}
\newcommand{\HpP}[1]{\text{$H^{#1}$}}

\begin{document}
\title[Data assimilation for the 2D B\'enard convection]{Continuous Data Assimilation for A 2D B\'enard Convection System through Horizontal Velocity Measurements Alone}

\date{January 28, 2016}

%
\author{Aseel Farhat}
\address[Aseel Farhat]{Department of Mathematics\\
               University of Virginia\\
       Charlottesville, VA 22904, USA}
\email[Aseel Farhat]{af7py@virginia.edu} 
\author{Evelyn Lunasin}
\address[Evelyn Lunasin]{Department of Mathematics\\
                United States Naval Academy\\
        Annapolis, MD, 21401 USA}
\email[Evelyn Lunasin]{lunasin@usna.edu}

\author{Edriss S. Titi}
\address[Edriss S. Titi]{Department of Mathematics, Texas A\&M University, 3368 TAMU,
 College Station, TX 77843-3368, USA.  {\bf ALSO},
  Department of Computer Science and Applied Mathematics, Weizmann Institute
  of Science, Rehovot 76100, Israel.} \email{titi@math.tamu.edu and
  edriss.titi@weizmann.ac.il}

\begin{abstract}
In this paper we propose a continuous data assimilation (downscaling) algorithm for a two-dimensional B\'enard convection problem.  Specifically we consider the two-dimensional Boussinesq system of a layer of incompressible fluid between two solid horizontal walls, with no-normal flow and stress free boundary condition on the walls, and fluid is heated from the bottom and cooled from the top. In this algorithm, we incorporate the observables as a feedback (nudging) term in the evolution equation of the {\it horizontal} velocity. We show that under an appropriate choice of the nudging parameter and the size of the spatial coarse mesh observables, and under the assumption that the observed data is error free, the solution of the proposed algorithm converges at an exponential rate, asymptotically in time, to the unique exact unknown reference solution of the original system, associated with the observed data on the horizontal component of the velocity.
Moreover, we note that in the case where the observational measurements are not error free, one can estimate the error between the solution of the algorithm and the exact reference solution of the system in terms of the error in the measurements.
\end{abstract}

 \maketitle
\section{Introduction}\label{intro}
The B\'enard convection problem is a model of the Boussinesq convection system of an incompressible fluid layer, confined between two solid walls,  which is heated from below in such a way that the lower wall maintains a temperature $T_0$ while the upper one maintains a temperature $T_1<T_0$. In this case, after some change of variables and proper scaling (by normalizing the distance between the walls and the temperature difference), the two-dimensional Boussinesq equations that govern the velocity, pressure and temperature are 

\begin{subequations}\label{Bous}
\begin{align}
&\pd{u}{t} - \nu\Delta u + (u\cdot\nabla)u + \nabla p' = \theta \mathbf{e}_2, \label{Bous1}\\
&\pd{\theta}{t} - \kappa\Delta\theta + (u\cdot\nabla)\theta - u\cdot{\bf e}_2=0, \label{Bous2}\\
&\nabla\cdot u= 0,\label{Bous_div}\\
&u(0;x_1,x_2) = u_0(x_1,x_2), \quad \theta(0;x_1,x_2)=\theta_0(x_1,x_2).\label{Bous_initial} 
\end{align}
In this paper, we will consider the above system with the following boundary conditions at walls:
\begin{align}
\text{No-normal flow at the wall: } u_2, \theta=0 \quad \text{at} \quad x_2=0 \quad \text{and} \quad x_2=1, \label{boundary1} \\
\text{Stress free at the wall: } \pd{u_1}{x_2}= 0 \quad \text{at} \quad x_2=0 \quad \text{and} \quad x_2=1, \label{boundary2}
\end{align}
and, for simplicity, we supplement the system with periodic boundary conditions in the horizontal direction: 
\begin{align}
& u, \theta, p \text{ are periodic, of period } L, \text{ in the }x_1\text{-direction}.\label{boundary3} 
\end{align}
\end{subequations}
Here, $u(t; x_1, x_2)=(u_1(t;x_1,x_2), u_2(t;x_1,x_2))$ is the fluid velocity, $p'=p'(t;x_1,x_2)$ is the modified pressure given by $p' = p - \left(x_2 +\frac{x_2^2}{2}\right)$ where $p=p(t;x_1,x_2)$ is the pressure of the fluid in the box $\Omega$, $\theta=\theta(t;x_1,x_2)$ is the scaled fluctuation of the temperature around the steady state background temperature profile $(T_1-T_0)x_2+T_0$  and it is given by $\theta = T- (\frac{T_0}{T_0-T_1}-x_2)$, where $T=T(t;x_1,x_2)$ is the temperature of the fluid inside the box $\Omega$, $\kappa$ and $\nu$ are the thermal diffusivity and kinematic viscosity, respectively.

The mathematical analysis of the B\'enard convection system \eqref{Bous} has been studied in \cite{Foias_Manley_Temam} (see also \cite{Temam_1997}), where the existence and uniqueness of weak solution in dimension two and three were proved, along with the existence of a finite-dimensional global attractor was also established in space dimension two. We remark that the analysis in \cite{Foias_Manley_Temam} was done considering the boundary conditions \eqref{boundary3} and Dirichlet boundary conditons for $u$ and $\theta$ at the top and the bottom boundaries. The authors in \cite{Foias_Manley_Temam} remarked that one can handle other natural boundary conditions following similar analysis by simply modifying the definition of the function spaces. Different cases for boundary conditions were discussed in \cite{Foias_Manley_Temam}. The special boundary conditions case we are considering above can be handled similarly. 

\subsection{Equivalent formulation of the B\'enard problem as a periodic boundary conditions}
Next, we will show that the solution of the initial-boundary value problem \eqref{Bous} is equivalent to the solution of system \eqref{Bous1}--\eqref{Bous2} subject to periodic boundary conditions with specific symmetry.

Consider any smooth solution $u =(u_1, u_2)$ of \eqref{Bous} and perform an even extension of the horizontal component $u_1$ across the boundary $x_2=0$: 
\begin{align}\label{sym_1}
u_1(x_1, x_2) = u_1(x_1,-x_2), \quad \text{for } (x_1,x_2) \in (0,L)\times(-1,0). 
\end{align}
This, with the divergence free condition, imposes an odd extension of the vertical component $u_2$ across the boundary $x_2=0$: 
\begin{align}\label{sym_2}
u_2(x_1,x_2) = -u_2(x_1,-x_2), \quad \text{for } (x_1,x_2) \in (0,L)\times(-1,0). 
\end{align}
The above extensions yield an odd extension on $\theta$ across the boundary $x_2=0$: 
\begin{align}\label{sym_3}
\theta(x_1,x_2) = -\theta(x_1,-x_2), \quad \text{for } (x_1,x_2) \in (0,L)\times(-1,0). 
\end{align}
With the view of the above extensions and the original boundary conditions \eqref{boundary1}--\eqref{boundary3}, we have 
\begin{subequations}\label{sym_4}
\begin{align}
u_2= \theta = 0, \quad \text{at } x_2 =-1,1, \label{sym_4a}
\end{align} and \begin{align}
\pd{u_1}{x_2} =0,  \quad \text{at } x_2 =-1,1. \label{sym_4b}
\end{align}
\end{subequations}

It is important to observe that this space of periodic functions with the specific symmetries \eqref{sym_1}--\eqref{sym_4} is invariant under the
solution operator of the B\'enard equations \eqref{Bous1}--\eqref{Bous_div} subject to periodic boundary conditions. It is also clear that such symmetric solutions satisfy the physical boundary conditions \eqref{boundary1}--\eqref{boundary3}. Hence all the results we obtain for the periodic case with
symmetry are equally valid for the physical problem. 

Based on this remark, in the rest of this paper, we will consider the B\'enard problem on the extended fully periodic domain $\Omega=(0,L)\times(-1,1)$ with the symmetries \eqref{sym_1}--\eqref{sym_4}. 

\subsection{A continuous data assimilation algorithm using horizontal velocity measurements only}


Designing feedback control algorithms for dynamical systems has been the focus of many authors in the past decades, see, e.g., \cite{CKTi, Leunberger1971, Thau1973, Nijmeijer2001} and references therein. In the context of meteorology and atmospheric physics, feedback control algorithms with a data assimilation prospective has been studied, e.g.,  in \cite{Ghil1977, Ghil-Halem-Atlas1978}. A finite-dimensional feedback control algorithm for stabilizing solutions of infinite-dimensional dissipative evolution equations, such as reaction-diffusion systems, the Navier-Stokes equations and the Kuramoto-Sivashinsky equation has been proposed and studied in \cite{Azouani_Titi, CKTi} (see also \cite{Lunasin_Titi}). Based on this control algorithm of \cite{Azouani_Titi}, a continuous data assimilation algorithm, where the coarse mesh observational measurements of the full state variables are incorporated into the equations in the form of a linear feedback control term, was developed in \cite{Azouani_Olson_Titi}. 
The algorithm was designed to work for general {\it linear and nonlinear} dissipative dynamical systems and it can be outlined as follows: consider a general dissipative evolutionary equation
\begin{align}\label{ev_eq}
\od{u}{t} = F(u),
\end{align}
where the initial data $u(0)= u_0$ is missing. The algorithm is of the form
\begin{subequations}\label{du}
\begin{align}
&\od{v}{t} = F(v) - \mu (I_h(v)- I_h(u)), \\
&v(0)= v_0,
\end{align}
\end{subequations}
where $\mu>0$ is a relaxation (nudging) parameter and $v^0$ is taken to be arbitrary initial data. $I_h(\cdot)$ represents an interpolant operator based on the observational measurements of a system at a coarse spatial resolution of size $h$, for $t\in [ 0,T ]$. Notice that if system \eqref{du} is globally well-posed and $I_h(v)$ converge to $I_h(u)$ in time, then we recover the reference $u(t,x)$ from the approximate solution $v(t,x)$. The main task is to find estimates on $\mu>0$ and $h>0$ such that the approximate solution $v(t)$ is with increasing accuracy to the reference solution $u(t)$ as more continuous data in time is supplied. Notice that this algorithm requires measurement of {\it all} state variables of the dynamical system \eqref{ev_eq}. 

The continuous data assimilation in the context of the incompressible 2D Navier-Stokes equations (NSE) was studied in \cite{Azouani_Olson_Titi} under the assumption that the data is noise free. A computational study of this algorithm was later presented in \cite{Gesho_Olson_Titi}. The case when the observational data contains stochastic noise is treated in \cite{Bessaih-Olson-Titi}. Most recently an extension of this algorithm for the case of discrete spatio-temporal measurements with error is studied in \cite{FMTi}; in addition, it is also shown there how to implement this algorithm in order to extract statistical properties of the relevant solution.

In \cite{MTT2015}, the authors analyzed an algorithm for continuous data assimilation for 3D Brinkman-Forchheimer-extended Darcy (3D BFeD) model of a porous medium, a model equation when the velocity is too large for classical Darcy's law to be valid. Furthermore, in \cite{ALT2014}, the proposed data assimilation algorithm was also applied to several three-dimensional subgrid scale turbulence models.

Analyzing the validity and success of a data assimilation algorithm when some state variable observations are not available is an important problem meteorology and engineering \cite{Charney1969, Hoke-Anthes1976, Ghil1977, Ghil-Halem-Atlas1978} (see also \cite{Altaf}). In a recent work \cite{FJT}, a continuous data assimilation scheme for the two-dimensional incompressible B\'enard convection problem was introduced. The data assimilation algorithm in \cite{FJT} constructs the approximate solutions for the velocity $u$ and temperature fluctuations $\theta$ using only the observational data, $I_h(u)$, of the velocity field and {\it without any measurements for the temperature (or density) fluctuations}. Inspired by the recent algorithm proposed in \cite{FJT}, we introduced an abridged dynamic continuous data assimilation for the 2D NSE in \cite{FLT2015}. The proposed algorithm in \cite{Azouani_Olson_Titi} for the 2D NSE requires measurements for the two components of the velocity vector field. On the other hand, in \cite{FLT2015}, we establish convergence results for an improved algorithm where the observational data needed to be measured and inserted into the model equation is reduced or subsampled. Our algorithm there requires observational measurements of {\it only one component} of the velocity vector field. An abridged data assimilation algorithm for the 3D Leray-$\alpha$ model, using observations in {\it only any two components and without any measurements on the third component of the velocity field}, was later analyzed in \cite{FLT2015_alpha}. In a more recent paper \cite{FLT_porous}, we proposed and analyzed a data assimilation algorithm for the 2D and 3D B\'enard convection problem in porous medium that employs measurements of the {\it temperature (or density) only}. 

Inspired the previous works \cite{FJT} and \cite{FLT2015}, in this paper we propose and analyze a continuous data assimilation algorithm for a B\'enard convection problem in the periodic box $\Omega = (0,L)\times(-1,1)$ with the symmetries \eqref{sym_1}--\eqref{sym_4}. Our algorithm requires measurements of {\it only the horizontal component of the velocity} to recover the full (velocity and temperature) reference solution of the 2D B\'enard convection problem. This algorithm is given by 
\begin{subequations}\label{DA_Bous}
\begin{align}
&\pd{v}{t} -\nu \Delta v + (v\cdot\nabla)v +\nabla \tilde p = \eta\mathbf{e}_2- \mu(I_h(v_1)-I_h(u_1)) {\bf e}_1,\\
&\pd{\eta}{t} -\kappa\Delta\eta - (v\cdot\nabla)\eta - v\cdot {\bf e}_2= 0, \\
&\nabla \cdot v = 0, \\
&v(0;x) = v_0(x), \quad \eta(0; x) = \eta_0(x), \label{DA_Bous_initial}
\end{align}
subject to the symmetries \eqref{sym_1}--\eqref{sym_2} on the solution $(v,\eta)$ with fully periodic boundary condition in the box $\Omega = (0,L)\times (-1,1)$. 
\end{subequations}
Here, $(v_0,\eta_0)$ can be taken arbitrary and $I_h$ is a linear interpolant operator of the measurements of the horizontal component of the velocity $u_1$. Two types of interpolants can be considered. 
One is to be given by a linear interpolant operator $I_h: \HpP{1} \rightarrow \LpP{2}$ satisfying
 the approximation property
\begin{align}\label{app}
\norm{\varphi - I_h(\varphi)}_{\LpP{2}}^2 \leq c_0h^2\norm{\varphi}_{\HpP{1}}^2, 
\end{align}
for every $\varphi \in \HpP{1}$, where $c_0>0$ is a dimensionless constant.
The other type is given by $I_h: \HpP{2}\rightarrow\LpP{2}$, together with
\begin{align}\label{app2}
\norm{\varphi - I_h(\varphi)}_{\LpP{2}}^2 \leq c_0h^2\norm{\varphi}_{\HpP{1}} + c_0^2h^4\norm{\varphi}_{\HpP{2}}^2,
\end{align}
for every $\varphi \in \HpP{2}$, where $c_0>0$ is a dimensionless constant. Notice that the interpolant operator $I_h(\varphi)$ (satisfying \eqref{app} or \eqref{app2}), with $\varphi$ satisfying the symmetry \eqref{sym_1}, will also satisfy the symmetry \eqref{sym_1}. 

One example of an interpolant observable that satisfies \eqref{app}, with $\varphi$ satisfying the symmetry \eqref{sym_1}, is the orthogonal projection onto the low Fourier modes with wave numbers $k$ such that $|k|\leq 1/h$. A more physical example is the volume elements that were studied in \cite{Azouani_Olson_Titi, Jones_Titi}. Examples of an interpolant observable that satisfies \eqref{app2} is given by the low Fourier modes and the measurements at a discrete set of nodal points in $\Omega$ (see Appendix A in \cite{Azouani_Olson_Titi}). Observe that one has to slightly modify the presentation given in \cite{Azouani_Olson_Titi, Jones_Titi} to fulfill the symmetry condition \eqref{sym_1}. 

In this paper, we give a rigorous justification that the solution of the data assimilation algorithm \eqref{DA_Bous},  $(v,\eta)$ is converging to the exact reference solution of the B\'enard convection problem $(u,\theta)$ subject to \eqref{sym_1}--\eqref{sym_4}. We provide explicit estimates on the relaxation (nudging) parameter $\mu$ and the spatial resolution $h$ of the observational measurements, in terms of physical parameters, that are needed in order for the proposed downscaling (data assimilation) algorithm to recover the reference solution under the assumption that the supplied data are error free.  An extension of algorithm \eqref{DA_Bous} for the case of discrete spatio-temporal measurements with error can be established by combining the ideas and tools reported in \cite{FMTi} with those presented here. 
\bigskip
\section{Preliminaries}\label{pre}

For the sake of completeness, this section presents some preliminary material and notation commonly used in the mathematical study of fluids, in particular in the study of the Navier-Stokes equations (NSE) and the Euler equations. For more detailed discussion on these topics, we refer the reader to \cite{Constantin_Foias_1988}, \cite{Robinson}, \cite{Temam_1995_Fun_Anal} and \cite{Temam_2001_Th_Num}.

We begin by defining function spaces corresponding to the relevant physical boundary conditions. We define $\mathcal{F}_1$ to be the set of trigonometric polynomials in $(x_1,x_2)$, with period $L$ in the $x_1$-direction, and are even with period $2$ in the $y$-direction. We define $\mathcal{F}_2$ to be the set of trigonometric polynomials in $(x_1,x_2)$ with period $L$ in the $x_1$-direction, and are odd with period $2$ in the $x_2$-direction. We denote the space of smooth vector-valued functions which incorporates the divergence-free condition by
\[\mathcal{V}:=\set{\phi\in\mathcal{F}_1\times\mathcal{F}_2: \; \nabla\cdot\phi=0}.\]

\begin{remark}
We will use the same notation indiscriminately for both scalar and vector Lebesgue and Sobolev spaces, which should not be a source of confusion.
\end{remark}

The closures of $\mathcal{V}$ and $\mathcal{F}_2$ in $L^2(\Omega)$ will be denoted by $H_0$ and $H_1$, respectively. $H_0$ and $H_1$ will be endowed  with the usual scalar product
\[(u,v)_{H_0}=\sum_{i=1}^2\int_{\Omega} u^i(x)v^i(x)\,dx
\sand
(\psi,\phi)_{H_1}=\int_{\Omega} \psi(x)\phi(x)\,dx, \]
and the associated norms $\norm{u}_{H_0} = (u,u)_{H_0}^{1/2}$ and $\norm{\phi}_{H_1} = (\phi,\phi)_{H_1}^{1/2}$, respectively. We denote the closures of $\mathcal{V}$ and $\mathcal{F}_2$ in $H^1(\Omega)$ by $V_0$ and $V_1$, respectively. $V_0$ and $V_1$ are Hilbert spaces endowed by the scalar product
\[((u,v))_{V_0}=(u,v)_{H_0} +\sum_{i,j=1}^2\int_{\Omega}\partial_ju_i(x)\partial_jv_i(x)\,dx
\] and 
\[((\psi,\phi))_{V_1}=\sum_{j=1}^2\int_{\Omega}\partial_j\psi(x)\partial_j\phi(x)\,dx, \]
and the associated norms $\norm{u}_{V_0} = ((u,u))_{V_0}^{1/2}$ and $\norm{\phi}_{V_1} = ((\phi,\phi))_{V_1}^{1/2}$, respectively.
\begin{remark} 
Since $\theta, u_2 \in V_1$ are zero at the boundary and are odd in the $x_2$-direction and periodic in the $x_1$-direction (with zero mean in the box), then by the Poincar\'e inequality \eqref{poincare}, $\norm{.}_{V_1}$ defines a norm on $V_1$. 
\end{remark} 

Let $D(A_0)= V_0\cap H^2(\Omega)$ and $D(A_1)= V_1\cap H^2(\Omega)$ and let $A_i: D(A_i) \rightarrow H_i$ be the unbounded linear operator defined by
\[ (A_i u, v)_{H_i} = ((u,v))_{V_i}, \qquad i = 0,1, \]
for all $u, v \in D(A_i)$. The operator $A_i$ is self-adjoint and $A_i^{-1}$ is a compact, non-negative, self-adjoint linear operator in $H_i$, for each $i=0,1$. Thus, there exists a complete orthonormal set of eigenfunctions $w_j^i$ in $H_i$ such that $A_iw_j^i= \lambda_j^iw_j^i$ where $0<\lambda_j^i\leq\lambda_{j+1}^i$ for $j\in \nN$ and each $i=0,1$. 
\begin{remark}\label{rmk23}
Notice that $A_1$ is a positive definite operator, while $A_0$ is a non-negative operator with finitely many eigenfunctions corresponding to the eigenvalue $\lambda =0$. Using Cauchy-Schwarz inequality and the elliptic regularity of the operator $A_0+I$, one can show that $\norm{u}_{H^2} \equiv \norm{u}_{L^2}+ \norm{A_0u}_{L^2}$. Moreover, in periodic boundary conditions, the operator $A_0 = -\Delta$. 
\end{remark}

We denote the Helmholtz-Leray projector from $L^2(\Omega)$ onto $H_0$ by $\cP_\sigma$ and the dual space of $V_i$ by $V_i^{'}$, for $i=0,1$. We define the bilinear map $B:V_0\times V_0 \rightarrow V_0^{'}$ by
$$\langle B(u,v), w\rangle_{V_0, V_0^{'}} = \int_\Omega(u\cdot\nabla)v,w\,dx, $$
for each $u, v,w \in V_0$,
and its scalar analogue
$\mathcal{B}:V_0\times V_1 \rightarrow V_1^{'}$ by
$$\langle\mathcal{B}(u,\theta), \phi\rangle_{V_1, V_1^{'}} = \int_\Omega(u\cdot\nabla)\theta,\phi\, dx, $$

These bilinear operators have the algebraic property
\begin{subequations}\label{prop1}
\begin{align}
\langle B(u,v), w\rangle_{V_0,V_0^{'}} = - \langle B(u,w),v\rangle_{V_0,V_0^{'}},\end{align}
and
\begin{align}
\langle \mathcal{B}(u,\theta),\phi\rangle_{V_1,V_1^{'}} = - \langle \mathcal{B}(u,\phi),\theta\rangle_{V_1,V_1^{'}},
\end{align}
\end{subequations}
for each $u \in V_0$ and $v, w \in V_0$ and $\theta, \phi \in V_1$. Consequently, the above bilinear maps also enjoy
the orthogonality property
\begin{align}\label{orth}
\langle B(u,v), v\rangle_{V_0,V_0^{'}} = 0, \qquad \text{and} \qquad \langle \mathcal{B}(u,\theta), \theta\rangle_{V_1,V_1^{'}} = 0,
\end{align}
for each $u\in V_0$, $v\in V_0$ and $\theta\in V_1$.
Also, in two dimensions (under periodic boundary conditions), the bilinear operator $B( ., .)$ satisfies the following identities (see, e.g. \cite{Temam_2001_Th_Num}):  
\begin{subequations}
\begin{align}\label{per_orth}
(B(u,u), A_0u)_{H_0}= 0, 
\end{align}
for each $u\in \mathcal{D}(A_0)$, and consequently 
\begin{align}\label{per_orth_2}
(B(u,w),A_0 w)_{H_0} + (B(w,u), A_0 w)_{H_0} + (B(w,w), A_0 u)_{H_0} = 0, 
\end{align}
\end{subequations}
for each $u$ and $w\in \mathcal{D}(A_0)$.

Employing the above notation, we write the incompressible two-dimensional B\'enard convection problem \eqref{Bous} in the functional form
\begin{subequations}\label{Bous_fun}
\begin{align}
&\od{u}{t} + \nu A_0u + B(u,u) = \cP_\sigma (\theta {\bf e}_2), \label{Bous_fun_1} \\
& \od{\theta}{t} + \kappa A_1 \theta + \mathcal{B}(u,\theta) - u \cdot{\bf e}_2= 0, \label{Bous_fun_2}\\
&u(0;x) = u_0(x), \quad \theta(0;x) = \theta_0(x).
\end{align}
\end{subequations}

Next, we recall the two-dimensional Ladyzhenskaya inequality 
\begin{equation}\label{L4_to_H1}
\|\vphi\|_{\LpP{4}}^2\leq c\|\vphi\|_{\LpP{2}}\norm{\vphi}_{\HpP{1}}, \quad \mbox{for every } \vphi \in \HpP{1}, 
\end{equation}
where $c$ is a  dimensionless, positive constant.
Hereafter, $c$ denotes a generic constant which may change from line to line. We also observe that we have the Poincar\'e inequality:
\begin{subequations}\label{poincare}
\begin{align}
\|\vphi\|_{\LpP{2}}^2\leq \lambda_1^{-1}\|\nabla \vphi\|_{\LpP{2}}^2, &\quad \text{ for all } \varphi\in {V_1},  \label{poincare_1}\\
\|\vphi\|_{V_1}^2\leq \lambda_1^{-1}\|A_1\vphi\|_{\LpP{2}}^2, &\quad \text{ for all } \varphi\in \mathcal{D}(A_1), \label{poincare_2}
\end{align}
\end{subequations}
where $\lambda_1$ is the smallest eigenvalues of the operator $A_1$.
\begin{remark}\label{poincare_velocity}
Notice that for $w =(w_1,w_2) \in V_0$ one has $w_2\in V_1$ and thus the Poincar\'e inequality is only valid for $w_2$, i.e. 
\begin{align*}
\|w_2\|_{\LpP{2}}^2\leq \lambda_1^{-1}\|\nabla w_2\|_{\LpP{2}}^2, \quad\text{and}\quad \|\nabla w_2\|_{\LpP{2}}^2\leq \lambda_1^{-1}\|A_1w_2\|_{\LpP{2}}^2= \lambda_1^{-1}\|\Delta w_2\|_{\LpP{2}}^2. 
\end{align*}
This is not the case for the horizontal component $w_1$. On the other hand, thanks to the boundary conditions and the divergence free condition, we have 
\begin{align*} 
\norm{w_i}_{H^2}^2 &= \norm{w_i}_{\LpP{2}}^2 + \norm{\nabla w_i}_{\LpP{2}}^2 + \sum_{j,k = 1}^2 \norm{\partial_j \partial_k w_i}_{\LpP{2}}^2 \notag \\
&\equiv \norm{w_i}_{\LpP{2}}^2 + \norm{\nabla w_i}_{\LpP{2}}^2 + \norm{\Delta w_i}_{\LpP{2}}^2, 
\end{align*}
for $i=1,2$. More precisely, $\norm{w_i}_{H^2}^2 = \norm{w_i}_{\LpP{2}}^2+ 2\norm{\Delta w_i}_{\LpP{2}}^2$, for $i=1,2$. 
 \end{remark}

Next, we will prove a lemma that we will use later in our analysis. 
\begin{lemma}\label{div_lemma} 
Let $u=(u_1, u_2) \in V_0$, then  
\begin{align}
\norm{u_2}_{\LpP{2}}^2 \leq \norm{\nabla u_1}_{\LpP{2}}^2. 
\end{align}
\end{lemma} 
\begin{proof}
Since $u\in V_0$, then  $u=(u_1, u_2) \in \HpP{1}$ and it satisfies the symmetries \eqref{sym_1} and \eqref{sym_2}, in particular, $u_2(x_1,0) =0$. Moreover, $u$  satisfies the divergence free condition $\nabla\cdot u=0$ in 
$\LpP{2}
$. Therefore, one has
\begin{align*}
u_2(x_1,x_2)&= u_2(x_1,0) + \int_0^{x_2} \pd{u_2}{x_2}(x_1,s)\,ds= -\int_0^{x_2} \pd{u_1}{x_1}(x_1,s)\,ds. 
\end{align*}
By Cauchy-Schwarz inequality, we get 
\begin{align*}
\abs{u_2(x_1,x_2)}&= \abs{\int_0^{x_2} \pd{u_1}{x_1}(x_1,s)\,ds}\leq \left( \int_0^1\abs{\pd{u_1}{x_1}(x_1,s)}^2\,ds\right)^{1/2}. 
\end{align*}
Thus, 
\begin{align*}
\int_0^L\abs{u_2(x_1,x_2)}^2\,dx_1 &\leq \int_0^L\int_0^1\abs{\pd{u_1}{x_1}(x_1,s)}^2\,dx_1 ds.
\end{align*}
This implies that 
\begin{align*}
\norm{u_2}_{\LpP{2}}^2 = \int_{-1}^1\int_0^L\abs{u_2(x_1,x_2)}^2\,dx_1ds &=2  \int_{0}^1\int_0^L\abs{u_2(x_1,x_2)}^2\,dx_1ds\notag \\
& \leq 2\int_0^L\int_0^1\abs{\pd{u_1}{x_1}(x_1,s)}^2\,dx_1 ds\notag \\
&  = \int_0^L\int_{-1}^1\abs{\pd{u_1}{x_1}(x_1,s)}^2\,dx_1 ds\leq \norm{\nabla u_1}_{\LpP{2}}^2. 
\end{align*}
\end{proof}
We will apply the following inequality which is a particular case of a more general inequality proved in \cite{Jones_Titi}.
\begin{lemma}\label{gen_gron_2}\cite{Jones_Titi} Let $\tau>0$ be fixed. Suppose that $Y(t)$ is an absolutely continuous function which is locally integrable and that it satisfies the following:
\begin{align*}
\od{Y}{t} + \alpha(t) Y \leq \beta(t),\qquad \text{ a.e. on } (0,\infty),
\end{align*}
such that 
\begin{align}\label{cond_1}
\liminf_{t\rightarrow\infty} \int_t^{t+\tau} \alpha(s)\,ds \geq \gamma, \qquad 
\limsup_{t\rightarrow\infty} \int_t^{t+\tau}  \alpha^{-}(s)\,ds < \infty,
\end{align}
and
\begin{align}
\lim_{t\rightarrow \infty} \int_t^{t+\tau} \beta^{+}(s)\,ds = 0, 
\end{align}
for some $\gamma>0$, where $\alpha^{-} = \max\{-\alpha, 0\}$ and $\beta^{+} = \max\{\beta,0\}$.
Then, $Y(t)\rightarrow 0$ at an exponential rate, as $t\rightarrow \infty$.
\end{lemma}
 
We also recall the following results from \cite{Foias_Manley_Temam, Temam_1997} for the B\'enard convection problem \eqref{Bous_fun}. These results were proved for a special case of boundary conditions: periodic in the $x_1$-direction and Dirichlet in the $x_2$-directions. The authors remarked that the analysis will follow similar steps for other natural boundary conditions. The same results hold for the boundary conditions we are considering in this paper: fully periodic boundary conditions with the symmetries \eqref{sym_1}--\eqref{sym_4}.   

\begin{theorem}[Existence and uniqueness of weak Solutions]
Let $T>0$ be fixed. Let $\nu>0$ and $\kappa>0$. If $u_0\in H_0$ and $\theta_0\in {H_1}$, then system \eqref{Bous_fun} has a unique weak solution $(u, \theta)$ such that $u\in C([0,T];H_0) \cap L^2([0,T];V_0)$ and $\theta\in C([0,T];{H_1}) \cap L^2([0,T], V_1)$. 
\end{theorem}

It was also shown in \cite{Foias_Manley_Temam, Temam_1997} that  the 2D B\'enard convection system has a finite-dimensional global attractor.

\begin{theorem}[Existence of a global attractor]\label{global_attractor_Bous} Let $T>0$ be fixed. If the initial data $u_0\in V_0$ and $\theta_0\in{V_1}$, then system \eqref{Bous_fun} has a unique strong solution $(u,\theta)$ that satisfies $u\in C([0,T];V_0)\cap L^2([0,T];\mathcal{D}(A_0))$ and $\theta\in C([0,T];{V_1})\cap L^2([0,T];\mathcal{D}(A_1))$. Moreover, system \eqref{Bous_fun} is globally well-posed and possesses a finite-dimensional global attractor $\mathcal{A}$ which is maximal among all the bounded invariant sets, and is compact in $H_0\times {H_1}$.
\end{theorem}

We will use the following bounds on $(u,\theta)$ later in our analysis. 
\begin{proposition}[A variant of the maximum principle]\label{max_principle}
Let $(u,\theta)$ be a strong solution of \eqref{Bous_fun},
then
\begin{align*}
\theta(t;\cdot) = \tilde{\theta}(t;\cdot) + \bar{\theta}(t;\cdot),
\end{align*}
where $-1\leq\tilde{\theta}(t;x)\leq1$ and
$$\norm{\bar{\theta}(t)}_{{H_1}} \leq \left(\norm{(\theta_0-1)_+}_{{H_1}} + \norm{(\theta_0+1)_-}_{{H_1}}\right)e^{-\kappa t},$$
for all $x\in\Omega$ and $t>0$. Here $M_+ = \max\{M,0\}$ and $M_{-}= \max\{-M,0\}$ for any real number $M$.
\end{proposition}

\begin{proposition}[Uniform bounds on the solutions]\label{unif_bounds}
Let $(u,\theta)$ be a strong solution of \eqref{Bous_fun}. There exists $t_0>0$, which depends on norms of the initial data, such that for all $t\geq t_0$,
\begin{align}
\norm{\theta(t)}_{{H_1}} &\leq a_0, \quad \text{and}\quad
\norm{u(t)}_{{H_0}}  \leq b_0, 
\end{align}
\begin{align}
&\int_t^{t+1}\norm{u(s)}_{{V_0}}^2\,ds \leq a_3, \quad
\int_t^{t+1}\norm{\theta(s)}_{{V_1}}^2\,ds \leq b_3,
\end{align}
\begin{align}
\norm{u(t)}_{{V_0}}^2 &\leq \left(a_2+a_3\right) e^{a_1} =: J_0, \label{J_0_eps}\\
\norm{\theta(t)}_{{V_1}}^2 & \leq (b_2+ b_3)e^{b_1} = : J_1, \label{J_1_eps}
\end{align}
where $a_0, a_1, a_2, a_3,, b_0, b_1, b_2, b_3, J_0$ and $J_1$ are positive constants that depend on $L, \nu, \mbox{and }  \kappa$. 
\end{proposition}
\bigskip 

\section{Convergence Results}\label{conv}
In this section, we derive conditions under which the approximate solution $(v,\eta)$, of the data assimilation algorithm system \eqref{DA_Bous_fun},  converges to the corresponding unique reference solution $(u,\theta)$ of the B\'enard convection problem \eqref{Bous_fun}, as $t\rightarrow \infty$. 


In functional form the data assimilation algorithm, system \eqref{DA_Bous}, reads as
\begin{subequations}\label{DA_Bous_fun}
\begin{align}
&\od{v}{t} + \nu A_0v + B(v,v) = \cP_\sigma(\eta {\bf e}_2) - \mu \cP_\sigma(I_h(v_1)-I_h(u_1)){\bf e}_1, \label{DA_v}\\ 
& \od{\eta}{t} +\kappa A_1\eta+ \mathcal{B}(v,\eta) - v\cdot{\bf e}_2= 0, \label{DA_e}\\
&v(0)= v_0 \quad \eta(0)= \eta_0. 
\end{align}
\end{subequations}
Here, $(u,\theta)$ is a strong solution of the 2D B\'enard convection problem \eqref{Bous_fun}, in the global attractor $\mathcal{A}$ corresponding to the observable measurements $I_h(u)$, that the above algorithm is designed to recover in a unique fashion.


\begin{theorem}\label{th_conv_1}
Suppose that $I_h$ satisfies the approximation property \eqref{app} and the symmetry property \eqref{sym_1}. Let $(u(t),\theta(t))$, for $t\geq 0$, be a strong solution in the global attractor of \eqref{Bous_fun}. 
\begin{enumerate} 
\item Let $T>0$, $v_0 \in V_0$, and $\eta_0 \in H_1$.  Suppose that $\mu>0$ is large enough such that 
\begin{align}\label{mu_1}
\mu\geq 2(K_1+ \nu),
\end{align}
where $K_1 = K_1(\nu, \kappa, L)$ is a constant defined in \eqref{K_1}, and $h>0$ is small enough such that $4\mu c_0^2 h^2\leq \nu $. Then, \eqref{DA_Bous_fun} has a unique solution $(v,\eta)$ that satisfies 
\begin{subequations}\label{weak}
\begin{align} 
v \in C([0,T];V_0)\cap L^2([0,T];\mathcal{D}(A_0)),\\
\eta\in C([0,T];{H_1})\cap L^2([0,T]; V_1), 
\end{align}
and also
\begin{align}
\od{v}{t} \in L^2([0,T];H_0), \qquad \od{\eta}{t}\in L^2([0,T];{V_1^{'}}). 
\end{align}
\end{subequations}
Moreover, the solution $(v,\eta)$ depends continuously on the initial data in the $V_0\times{H_1}$ norm, and it satisfies $$\norm{u(t)-v(t)}_{V_0}^2 + \norm{\theta(t)-\eta(t)}_{H_1}^2\rightarrow 0,$$ at an 
exponential rate, as $t \rightarrow \infty$.

\item Let $T> 0$, $v_0\in V_0$, and $\eta_0\in V_1$. Suppose that $\mu>0$ is large enough such that 
\begin{align}\label{mu_2}
\mu \geq 2(K_1+\nu) + 2K_2, 
\end{align}
where $K_{i} = K_{i}(\nu, \kappa, L)$, $i=1,2$, are constants defined in \eqref{K_1} and \eqref{K_2}, respectively, and suppose that $h>0$ is small enough such that $4\mu c_0^2 h^2\leq \nu $. Then, \eqref{DA_Bous_fun} has a unique strong solution $(v,\eta)$ that satisfies 
\begin{subequations}\label{strong}
\begin{align} 
v \in C([0,T];V_0)\cap L^2([0,T];\mathcal{D}(A_0)),\\
\eta\in C([0,T];V_1)\cap L^2([0,T]; \mathcal{D}(A_1)), 
\end{align}
and 
\begin{align}
\od{v}{t} \in L^2([0,T];H_0), \qquad \od{\eta}{t}\in L^2([0,T];H_1). 
\end{align}
\end{subequations}
Moreover, the strong solution $(v,\eta)$ depends continuously on the initial data, in the $V_0\times{V_1}$ norm, and it satisfies 
$$\norm{u(t)-v(t)}_{V_0}^2 + \norm{\theta(t)-\eta(t)}_{V_1}^2\rightarrow 0,$$ at an 
exponential rate, as $t \rightarrow \infty$.  
\end{enumerate}
\end{theorem}

\begin{proof}
Since we assume that $(u,\theta)$ is a reference solution of system \eqref{Bous}, then it is enough to show the existence and uniqueness of the difference $(w,\xi)=(u-v, \theta-\eta)$. In the proof below, we will drive formal \textit{a-priori} bounds on the difference $(w, \xi)$, under the conditions that $\mu$ is large enough and $h$ is small enough such that $4\mu c_0^2h \leq \nu$. These \textit{a-priori} estimates, together with the global existence and uniqueness of the solution $(u,\theta)$, form the key elements for showing the global existence of the solution $(v,\eta)$ of system \eqref{DA_Bous}. The convergence of the approximate solution $(v,\eta)$ to the exact reference solution $(u,\theta)$ will also be established under the tighter condition on the nudging parameter $\mu$ as stated in \eqref{mu_1}. Uniqueness can then be obtained using similar energy estimates.

The estimates we provide in this proof are formal, but can be justified by the Galerkin approximation procedure
and then passing to the limit while using the relevant compactness theorems. We
will omit the rigorous details of this standard procedure (see, e.g., \cite{Constantin_Foias_1988, Robinson, Temam_2001_Th_Num}) and provide only the formal \textit{a-priori} estimates.
 
As above we define $w = u-v$, $\xi = \theta-\eta$. Then $w$ and $\xi$ satisfy the system\begin{subequations}
\begin{align}
&\od{w}{t} +\nu A_0w +B(v,w)+ B(w,u) = \cP_\sigma(\xi \mathbf{e}_2)- \mu \cP_\sigma (I_h(w_1){\bf e}_1), \label{w}\\
& \od{\xi}{t} - \kappa A_1 \xi +\mathcal{B}(v,\xi) + \mathcal{B}(w,\theta)- w\cdot{\bf e}_2= 0, \label{xi}\\
&w(0) = w_0: = u_0-v_0, \\
&\xi(0) = \xi_0 := \theta_0 - \eta_0. 
\end{align}
\end{subequations}
Taking the ${\LpP{2}}$ inner product of \eqref{w} and \eqref{xi} with $w$ and $\xi$, respectively, we obtain 
\begin{subequations}\label{ode_L2}
\begin{align}
&\frac 12 \od{}{t} \norm{w}_{\LpP{2}}^2 + \nu \norm{\nabla w}_{\LpP{2}}^2 + \left(B(w,u),w\right) = \int_{\Omega} \xi (w\cdot{\bf e}_2)\, dx - \mu (I_h(w_1), w_1), \\
&\frac1 2 \od{}{t} \norm{\xi}_{\LpP{2}}^2 + \kappa\norm{\nabla \xi}_{\LpP{2}}^2 + \left(\mathcal{B}(w,\theta),\xi\right) = \int_{\Omega} \xi (w\cdot{\bf e}_2)\, dx. \label{xi_ode_L2}
\end{align}
\end{subequations}
By H\"older and Young inequalities, Lemma \ref{div_lemma}, and Poincar\'e inequality \eqref{poincare}, we have 
\begin{align}\label{1}
\abs{\int_{\Omega}\xi (w\cdot{\bf e}_2)\, dx} &\leq \norm{w_2}_{\LpP{2}}\norm{\xi}_{\LpP{2}}\notag \\
& \leq \frac{\kappa\lambda_1}{20}\norm{\xi}_{{\LpP{2}}}^2 + \frac{c}{\kappa\lambda_1}\norm{w_2}_{{\LpP{2}}}^2\notag \\
&\leq \frac{\kappa}{20}\norm{\nabla\xi}_{{\LpP{2}}}^2 +  \frac{c}{\kappa\lambda_1}\norm{\nabla w_1}_{\LpP{2}}^2.   
\end{align}
Young inequality and Lemma \ref{div_lemma} yield 
\begin{align}\label{2}
\abs{\left(\mathcal{B}(w,\theta), \xi\right)} &= \abs{\left(\mathcal{B}(w,\xi), \theta\right)} \notag \\
&\leq \norm{\theta}_{\LpP{\infty}}\norm{w}_{\LpP{2}}\norm{\nabla \xi}_{\LpP{2}}\notag \\
& \leq \frac{\kappa}{20}\norm{\nabla\xi}_{\LpP{2}}^2 + \frac{c}{\kappa}\norm{\theta}_{\LpP{\infty}}^2\norm{w}_{\LpP{2}}^2 \notag\\
& \leq \frac{\kappa}{20}\norm{\nabla\xi}_{\LpP{2}}^2 + \frac{c}{\kappa}\norm{\theta}_{\LpP{\infty}}^2\left(\norm{w_1}_{\LpP{2}}^2 + \norm{\nabla w_1}_{\LpP{2}}^2\right). 
\end{align}
Also (thanks to Ladyzhanskaya inequality \eqref{L4_to_H1})
\begin{align}\label{2}
\abs{\left(B(w,u),w\right)} &\leq \norm{\nabla u}_{\LpP{2}}\norm{w}_{\LpP{4}}^2 \notag \\
&\leq c\norm{\nabla u}_{\LpP{2}}\norm{w}_{\LpP{2}}\norm{w}_{{V_0}}\notag \\
&\leq c \norm{\nabla u}_{\LpP{2}}\norm{w}_{\LpP{2}}\left(\norm{w}_{\LpP{2}}^2 + \norm{\nabla w}_{\LpP{2}}^2\right)^{1/2}\notag\\
&\leq c\norm{\nabla u}_{\LpP{2}}\norm{w}_{\LpP{2}}^2 + c \norm{\nabla u}_{\LpP{2}}\norm{w}_{\LpP{2}}\norm{\nabla w}_{\LpP{2}}\notag \\
& \leq  c\norm{\nabla u}_{\LpP{2}}\norm{w}_{\LpP{2}}^2 +\frac {\nu}{20} \norm{\nabla w}_{\LpP{2}}^2 +\frac {c}{\nu} \norm{\nabla u}_{\LpP{2}}^2\norm{w}_{\LpP{2}}^2 \notag\\
&\leq \frac{\nu}{20}\norm{\nabla w}_{\LpP{2}}^2 + c \norm{\nabla u}_{\LpP{2}}\left(1 + \frac{\norm{\nabla u}_{\LpP{2}}}{\nu}\right)\left(\norm{w_1}_{\LpP{2}}^2 + \norm{\nabla w_1}_{\LpP{2}}^2\right). 
\end{align}
Thanks to the assumptions $4\mu c_0^2h^2\leq \nu$ and \eqref{mu_1}, and Young inequality, we have 
\begin{align}\label{4}
-\mu(I_h(w_1),w_1) & = -\mu(I_h(w_1)-w_1, w_1) - \mu \norm{w_1}_{\LpP{2}}^2\notag \\
& \leq \mu \norm{I_h(w_1)-w_1}_{\LpP{2}}\norm{w_1}_{\LpP{2}} - \mu \norm{w_1}_{\LpP{2}}^2\notag \\
&\leq \mu c_0h\norm{w_1}_{\LpP{2}}\norm{w_1}_{\HpP{1}} - \mu \norm{w_1}_{\LpP{2}}^2\notag \\
& \leq \mu c_0^2h^2\norm{w_1}_{\HpP{1}}^2 - \frac{3\mu}{4}\norm{w_1}_{\LpP{2}}^2\notag \\
&\leq \frac{\nu}{4}\left(\norm{w_1}_{\LpP{2}}^2 +\norm{\nabla w_1}_{\LpP{2}}^2\right) -\frac{3\mu}{4}\norm{w_1}_{\LpP{2}}^2\notag \\
&\leq \frac{\nu}{4}\norm{\nabla w_1}_{\LpP{2}}^2 -\frac{5\mu}{8}\norm{w_1}_{\LpP{2}}^2.
\end{align}

Taking the ${\LpP{2}}$-inner product of \eqref{w} with $A_0w=-\Delta w$, and using the orthogonality properties \eqref{per_orth} and \eqref{per_orth_2}, we have
\begin{align}\label{ode_H1}
\frac 12 \od{}{t} \norm{\nabla w}_{\LpP{2}}^2+ \nu \norm{A_0w}_{\LpP{2}}^2 +\left(B(w,w),A_0u\right) &= \int_{\Omega} \xi (A_0w\cdot{\bf e}_2)\, dx \notag \\ &\quad - \mu (I_h(w_1), \Delta w_1).
\end{align} 
Using H\"older inequality and Ladyzenskaya inequality \eqref{L4_to_H1}, we get 
\begin{align*}
\abs{\left(B(w,w),A_0u\right)}&\leq \norm{A_0u}_{\LpP{2}} \norm{w}_{\LpP{4}}\norm{\nabla w}_{\LpP{4}}\notag \\
&\leq c\norm{A_0u}_{\LpP{2}}\norm{w}_{\LpP{2}}^{1/2}\norm{w}_{V_0}\norm{w}_{\HpP{2}}^{1/2}. 
\end{align*}
Thanks to Remark \ref{rmk23}, we have 
\begin{align*}
&\abs{\left(B(w,w),A_0u\right)}\leq c\norm{A_0u}_{\LpP{2}}\norm{w}_{\LpP{2}}^{1/2}\left(\norm{w}_{\LpP{2}}^2 + \norm{\nabla w}_{\LpP{2}}^2\right)^{1/2}\left(\norm{w}_{\LpP{2}}^{2} + \norm{A_0 w}_{\LpP{2}}^{2}\right)^{1/4}\notag \\
& \leq c\norm{A_0u}_{\LpP{2}}\norm{w}_{\LpP{2}}^{1/2}\left(\norm{w}_{\LpP{2}}+ \norm{\nabla w}_{\LpP{2}}\right)\left(\norm{w}_{\LpP{2}}^{1/2} + \norm{A_0 w}_{\LpP{2}}^{1/2}\right)\notag \\
& \leq c \norm{A_0 u}_{\LpP{2}} \left(\norm{w}_{\LpP{2}}^2 + \norm{w}_{\LpP{2}}\norm{\nabla w}_{\LpP{2}} + \norm{w}_{\LpP{2}}^{3/2}\norm{A_0w}_{\LpP{2}}^{1/2} + \norm{w}_{\LpP{2}}^{1/2}\norm{\nabla w}_{\LpP{2}}\norm{A_0 w}_{\LpP{2}}^{1/2}\right).
\end{align*} 
Since $A_0w = -\Delta w$ in periodic boundary conditions, we also have 
\begin{align*}
\norm{\nabla w}_{\LpP{2}}^2 = \int_\Omega \nabla w\cdot \nabla w\, dxdy &= \int_\Omega wA_0 w\, dxdy \notag \\
& \leq \norm{w}_{\LpP{2}}\norm{A_0 w}_{\LpP{2}}. 
\end{align*} 
Thus, Young inequality and Lemma \ref{div_lemma} imply 
\begin{align}\label{1b}
&\abs{\left(B(w,w),A_0u\right)} \leq c\norm{A_0 u}_{\LpP{2}}\left(\norm{w}_{\LpP{2}}^2 + \norm{w}_{\LpP{2}}^{3/2}\norm{A_0w}_{\LpP{2}}^{1/2}+ \norm{w}_{\LpP{2}}\norm{A_0w}_{\LpP{2}}\right)\notag \\
&\quad  \leq \frac{\nu}{20}\norm{A_0w}_{\LpP{2}}^2 + c \norm{A_0 u}_{\LpP{2}}\left(1+\frac{\norm{A_0u}_{\LpP{2}}^{1/3}}{\nu^{1/3}}+ \frac{\norm{A_0u}_{\LpP{2}}}{\nu}\right)\norm{w}_{\LpP{2}}^2\notag\\
& \quad \leq \frac{\nu}{20}\norm{A_0w}_{\LpP{2}}^2 + c \norm{A_0 u}_{\LpP{2}}\left(1+\frac{\norm{A_0u}_{\LpP{2}}^{1/3}}{\nu^{1/3}}+ \frac{\norm{A_0u}_{\LpP{2}}}{\nu}\right)\left(\norm{w_1}_{\LpP{2}}^2 + \norm{\nabla w_1}_{\LpP{2}}^2\right). 
\end{align} 

Also, Young inequality and Lemma \ref{div_lemma} yield 
\begin{align}\label{2b}
\abs{\int_{\Omega} \xi (A_0w\cdot{\bf e}_2)\, dx} &= \abs{\int_{\Omega} \xi \Delta w_2\, dx} \notag \\
&\leq \norm{\nabla w_2}_{\LpP{2}}\norm{\nabla\xi}_{\LpP{2}}\notag \\
&\leq \frac{\kappa}{20}\norm{\nabla \xi}_{\LpP{2}}^2 + \frac{c}{\kappa}\norm{\nabla w_2}_{\LpP{2}}^2\notag \\
& \leq \frac{\kappa}{20}\norm{\nabla \xi}_{\LpP{2}}^2 + \frac{c}{\kappa}\norm{w_2}_{\LpP{2}}\norm{\Delta w_2}_{\LpP{2}}\notag \\
&\leq  \frac{\kappa}{20}\norm{\nabla \xi}_{\LpP{2}}^2 + \frac{\nu}{20}\norm{\Delta w_2}_{\LpP{2}}^2 + \frac{c}{\nu\kappa^2}\norm{w_2}_{\LpP{2}}^2 \notag \\
& \leq  \frac{\kappa}{20}\norm{\nabla \xi}_{\LpP{2}}^2 + \frac{\nu}{20}\norm{\Delta w_2}_{\LpP{2}}^2 + \frac{c}{\nu\kappa^2}\norm{\nabla w_1}_{\LpP{2}}^2. 
\end{align}

Using \eqref{app}, Young inequality and the assumption that $4\mu ch^2 \leq \nu$, we have
\begin{align}\label{3b}
-\mu (I_h(w_1), -\Delta w_1) &= \mu (I_h(w_1) - w_1  , \Delta w_1) - \mu \norm{\nabla w_1}_{\LpP{2}}^2\notag \\
&\leq \mu c_0h\norm{w_1}_{\HpP{1}}\norm{\Delta w_1}_{\LpP{2}} - \mu \norm{\nabla w_1}_{\LpP{2}}^2\notag \\
&\leq \frac{\mu^2ch^2}{2\nu} \left(\norm{w_1}_{\LpP{2}}^2 + \norm{\nabla w_1}_{\LpP{2}}^2 \right) + \frac{\nu}{2}\norm{\Delta w_1}_{\LpP{2}}^2 - \mu \norm{\nabla w_1}_{\LpP{2}}^2\notag \\
& \leq \frac{\nu}{2} \norm{\Delta w_1}_{\LpP{2}}^2 +\frac{\mu}{8}\norm{w_1}_{\LpP{2}}^2 - \frac{7\mu}{8} \norm{\nabla w_1}_{\LpP{2}}^2.   
\end{align} 

Thanks to the Poincar\'e inequality \eqref{poincare_1} (see Remark \ref{poincare_velocity}) we have $\norm{\nabla w_2}_{\LpP{2}}^2 \geq \lambda_1 \norm{w_2}_{\LpP{2}}^2$ and $\norm{\nabla \xi}_{\LpP{2}}^2 \geq \lambda_1 \norm{\xi}_{\LpP{2}}^2$. This implies that 
\begin{align}\label{poincare_estimate_1}
\nu \norm{\nabla w}_{\LpP{2}}^2 + \kappa \norm{\nabla \xi}_{\LpP{2}}^2&= \nu\norm{\nabla w_1}_{\LpP{2}}^2 + \frac \nu 2 \norm{\nabla w_2}_{\LpP{2}}^2 + \frac \nu 2 \norm{\nabla w_2}_{\LpP{2}}^2+ \kappa \norm{\nabla \xi}_{\LpP{2}}^2 \notag \\
& \geq \nu \norm{\nabla w_1}_{\LpP{2}}^2 +  \frac \nu 2\norm{\nabla w_2}_{\LpP{2}}^2+ \frac{\nu\lambda_1}{2} \norm{w_2}_{\LpP{2}}^2+ \kappa \lambda_1\norm{\xi}_{\LpP{2}}^2.
\end{align} 
Since $\nu \norm{A_0 w}_{\LpP{2}}^2 \geq 0$, it follows from equations \eqref{ode_L2} and \eqref{ode_H1} and estimates \eqref{1}--\eqref{4} and \eqref{1b}--\eqref{3b} and \eqref{poincare_estimate_1}: 
\begin{align}\label{ode_conv_1}
&\od{}{t}\left(\norm{w }_{V _0}^2 + \norm{\xi}_{\LpP{2}}^2\right) +\frac{\min\left\{\nu,\kappa\right\}}{8}\left(\norm{\nabla w_1}_{\LpP{2}}^2 + \frac{\norm{\nabla w_2}_{\LpP{2}}^2}{2} +\frac{\lambda_1}{2}\norm{w_2}_{\LpP{2}}^2 + \lambda_1\norm{\xi}_{\LpP{2}}^2\right)\notag\\
&\qquad \leq \od{}{t}\left(\norm{w }_{V _0}^2 + \norm{\xi}_{\LpP{2}}^2\right) +\frac{\min\left\{\nu,\kappa\right\}}{8}\left(\norm{\nabla w}_{\LpP{2}}^2 +\norm{\nabla \xi}_{\LpP{2}}^2\right) \notag \\
& \qquad\leq (\alpha(t)-\mu)\left(\norm{w_1}_{\LpP{2}}^2 + \norm{\nabla w_1}_{\LpP{2}}^2\right), 
\end{align} 
where 
\begin{align}\label{alpha}
\alpha(t) &:= \frac{c}{\kappa \lambda_1} + \frac{c}{\nu\kappa^2} + \frac{c}{\kappa}\norm{\theta(t)}_{\LpP{\infty}}^2 + c \norm{\nabla u(t)}_{\LpP{2}}\left(1 + \frac{\norm{\nabla u(t)}_{\LpP{2}}}{\nu}\right) \notag\\& \quad + c \norm{A_0 u(t)}_{\LpP{2}}\left(1+\frac{\norm{A_0u(t)}_{\LpP{2}}^{1/3}}{\nu^{1/3}}+ \frac{\norm{A_0u(t)}_{\LpP{2}}}{\nu}\right). 
\end{align}

Since by assumption $(u,\theta)$ is a solution that is contained in the global attractor of \eqref{DA_Bous_fun}, by Proposition \ref{max_principle} and Proposition \ref{unif_bounds}, we conclude that there exist a positive constants $K_1= K_1(\nu,\kappa,\lambda_1,L)$ 
such that for all $t\in \mathbb{R}$
\begin{align}\label{K_1}
\alpha(t) \leq K_1. 
\end{align}
Then, assumption \eqref{mu_1}
implies that $\mu - \alpha(t) \geq \frac{\mu}{2}$, for all $t\geq 0$. 
Thus, thanks to the Poincar\'e inequality \eqref{poincare}, inequality \eqref{ode_conv_1} implies
\begin{align}\label{ode_conv_2}
\od{}{t} \left(\norm{w}_{V_0}^2 + \norm{\xi}_{H_1}^2\right) + \gamma\left(\norm{w}_{V_0}^2 + \norm{\xi}_{H_1}^2\right)\leq 0,
\end{align}
where $\gamma =  \min\left\{\frac{\nu}{16},\frac{\nu\lambda_1}{16},\frac{\kappa\lambda_1}{8},\frac{\mu}{2}\right\}$. 
By Gronwall's inequality, it follows that
\begin{align}\label{asym_bound}
\norm{w(t)}_{V_0}^2 + \norm{\xi(t)}_{H_1}^2 \leq \left(  \norm{w(0)}_{V_0}^2 + \norm{\xi(0)}_{H_1}^2\right) e^{-\gamma t}, 
\end{align}
for every $t\geq 0$. 
Next we prove the second part of the theorem. \\

Taking the $\LpP{2}$-inner product of \eqref{xi} with $A_1\xi = -\Delta\xi$, we have 
\begin{align}\label{ode2}
\frac{1}{2}\od{}{t}\norm{\nabla \xi}_{\LpP{2}}^2 + \kappa\norm{\Delta \xi}_{\LpP{2}}^2 +\left(\mathcal{B}(w,\theta),-\Delta \xi\right) + \left(\mathcal{B}(v,\xi),-\Delta \xi\right) \leq (w_2, -\Delta \xi).
\end{align}
Using  Cauchy-Schwarz and Young inequality, and Lemma \ref{div_lemma}, we have
\begin{align}\label{1c}
\abs{(w_2, -\Delta \xi)} \leq  \norm{w_2}_{\LpP{2}}^2 \norm{\Delta\xi}_{\LpP{2}}^2 &\leq\frac{\kappa}{20}\norm{\Delta\xi}_{\LpP{2}}^2 + \frac{c}{\kappa}\norm{w_2}_{\LpP{2}}^2 \notag\\
& \leq \frac{\kappa}{20}\norm{\Delta\xi}_{\LpP{2}}^2 +\frac{c}{\kappa}\norm{\nabla w_1}_{\LpP{2}}^2 . 
\end{align}
Using H\"older inequality, Ladyzhenskaya inequality \eqref{L4_to_H1} and Lemma \ref{div_lemma}, we get
\begin{align}\label{2c}
\abs{\left(\mathcal{B}(w,\theta),-\Delta \xi\right)}&\leq \norm{w}_{\LpP{4}}\norm{\nabla \theta}_{\LpP{4}}\norm{\Delta \xi}_{\LpP{2}}\notag \\ 
& \leq \frac{\kappa}{20}\norm{\Delta\xi}_{\LpP{2}}^2 + \frac{c}{\kappa}\norm{w}_{\LpP{4}}^2 \norm{\nabla\theta}_{\LpP{4}}^2\notag\\
& \leq \frac{\kappa}{20}\norm{\Delta\xi}_{\LpP{2}}^2 + \frac{c}{\kappa}\norm{w}_{\LpP{2}}\norm{w}_{H^1}\norm{\nabla\theta}_{\LpP{2}}\norm{\Delta\theta}_{\LpP{2}}\notag \\
& \leq \frac{\kappa}{20}\norm{\Delta \xi}_{\LpP{2}}^2 + \frac{c}{\kappa}\norm{w}_{\LpP{2}}\left(\norm{w}_{\LpP{2}}^2+\norm{\nabla w}_{\LpP{2}}^2\right)^{1/2}\norm{\nabla\theta}_{\LpP{2}}\norm{\Delta\theta}_{\LpP{2}}\notag \\
& \leq \frac{\kappa}{20}\norm{\Delta\xi}_{\LpP{2}}^2 + \frac{c}{\kappa}\left(\norm{w}_{\LpP{2}}^2+\norm{\nabla w}_{\LpP{2}}\norm{w}_{\LpP{2}}\right)\norm{\nabla\theta}_{\LpP{2}}\norm{\Delta\theta}_{\LpP{2}}\notag \\
& \leq \frac{\kappa}{20}\norm{\Delta\xi}_{\LpP{2}}^2 + \frac{\nu}{20}\norm{\nabla w}_{\LpP{2}}^2+ \frac{c}{\kappa}\norm{\nabla\theta}_{\LpP{2}}\norm{\Delta\theta}_{\LpP{2}}\left(1+\norm{\nabla\theta}_{\LpP{2}}\norm{\Delta\theta}_{\LpP{2}}\right)\norm{w}_{\LpP{2}}^2\notag \\
&\leq \frac{c}{\kappa}\norm{\nabla\theta}_{\LpP{2}}\norm{\Delta\theta}_{\LpP{2}}\left(1+\norm{\nabla\theta}_{\LpP{2}}\norm{\Delta\theta}_{\LpP{2}}\right)\left(\norm{w_1}_{\LpP{2}}^2+\norm{\nabla w_1}_{\LpP{2}}^2\right) \notag \\
&\qquad +\frac{\kappa}{20}\norm{\Delta \xi}_{\LpP{2}}^2 + \frac{\nu}{20}\norm{\nabla w}_{\LpP{2}}^2.
\end{align}
The H\"older inequality and Ladyzhenskaya inequality \eqref{L4_to_H1} also yield 
\begin{align}\label{3c}
\abs{\left(\mathcal{B}(v,\xi),\Delta \xi)\right)} &\leq \norm{v}_{\LpP{4}}\norm{\nabla \xi}_{\LpP{4}}\norm{\Delta \xi}_{\LpP{2}}\notag \\
& \leq c\norm{v}_{\LpP{2}}^{1/2}\norm{v}_{V_0}^{1/2}\norm{\nabla \xi}_{\LpP{2}}^{1/2}\norm{\Delta\xi}_{\LpP{2}}^{3/2}\notag \\
& \leq \frac{\kappa}{20}\norm{\Delta \xi}_{\LpP{2}}^2 + c\norm{v}_{\LpP{2}}^2\norm{v}_{V_0}^2\norm{\nabla \xi}_{\LpP{2}}^2. 
\end{align}

Since $\nu \norm{A_0 w}_{\LpP{2}}^2 \geq 0$, we conclude from equations \eqref{ode_L2}, \eqref{ode_H1} and \eqref{ode2}, and estimates \eqref{1}--\eqref{4}, \eqref{1b}--\eqref{3b}, and \eqref{1c}--\eqref{3c}, that
\begin{align}\label{ode_V}
&\od{}{t}\left(\norm{w }_{V _0}^2 + \norm{\nabla\xi}_{\LpP{2}}^2\right)+\frac{\min\left\{\nu,\kappa\right\}}{8}\left(\norm{\nabla w}_{\LpP{2}}^2 + \norm{\Delta \xi}_{\LpP{2}}^2\right)\notag\\
& \quad \leq (\tilde{\alpha}(t)-\mu)\left(\norm{w_1}_{\LpP{2}}^2 + \norm{\nabla w_1}_{\LpP{2}}^2\right) +2\norm{v}_{\LpP{2}}^2\norm{v}_{V_0}^2\norm{\nabla \xi}_{\LpP{2}}^2, 
\end{align}
where 
\begin{align}\label{alpha_prime}
\tilde{\alpha}(t) &:= \frac{c}{\kappa \lambda_1} + \frac{c}{\kappa} + c \norm{\nabla u(t)}_{\LpP{2}}\left(1 + \frac{\norm{\nabla u(t)}_{\LpP{2}}}{\nu}\right) \notag\\& \quad + c \norm{A_0 u(t)}_{\LpP{2}}\left(1+\frac{\norm{A_0u(t)}_{\LpP{2}}^{1/3}}{\nu^{1/3}}+ \frac{\norm{A_0u(t)}_{\LpP{2}}}{\nu}\right) \notag\\ &\quad + \frac{c}{\kappa}\norm{\nabla\theta}_{\LpP{2}}\norm{\Delta\theta}_{\LpP{2}}\left(1+\norm{\nabla\theta}_{\LpP{2}}\norm{\Delta\theta}_{\LpP{2}}\right) \notag \\
& \leq \alpha(t) + \frac{c}{\kappa}\norm{\nabla\theta}_{\LpP{2}}\norm{\Delta\theta}_{\LpP{2}}\left(1+\norm{\nabla\theta}_{\LpP{2}}\norm{\Delta\theta}_{\LpP{2}}\right), 
\end{align}
where $\alpha(t)$ is defined in \eqref{alpha}. Since $(u,\theta)$ is the reference solution is assumed to be contained in the global attractor of \eqref{Bous_fun}, then by Proposition \ref{unif_bounds}, there exists a constant $K_2=K_2(\nu,\kappa,\lambda_1,L)$ such that, for all $t\in\nR$, 
\begin{align}\label{K_2} 
\tilde{\alpha}(t) \leq K_1 + K_2,
\end{align}
where $K_1$ is a constant defined in \eqref{K_1}. 
By the first part of the theorem, $v(t)$ is a global solution of \eqref{DA_Bous} belongs to $C([0,T], V_0)$ for any $T>0$. Moreover, assumption \eqref{mu_2} implies that $\norm{u(t)-v(t)}_{V_0}^2 \rightarrow 0$, as $t\rightarrow\infty$. Then, by Proposition \ref{unif_bounds}
\begin{align}\label{K_3}
\norm{v(t)}_{V_0}^2 \leq K_3, 
\end{align} 
for some constant $K_3= K_3(\nu,\kappa,L)$, for all $t\geq 0$.

Now, assumption \eqref{mu_2}, equation \eqref{ode_V} and estimates \eqref{alpha_prime}--\eqref{K_3} yield that $\mu-\alpha(t)\geq \frac{\mu}{2}>0$, for $t\geq 0$, and thus
\begin{align}
\od{}{t}\left(\norm{w }_{V _0}^2 + \norm{\nabla\xi}_{\LpP{2}}^2\right) &+\frac{\min\left\{\nu,\kappa\right\}}{8}\left(\norm{\nabla w}_{\LpP{2}}^2 + \norm{\Delta \xi}_{\LpP{2}}^2\right) \notag\\
& + \frac{\mu}{2}\left(\norm{w_1}_{\LpP{2}}^2 + \norm{\nabla w_1}_{\LpP{2}}^2\right) \leq cK_3^2 \norm{\nabla\xi}_{\LpP{2}}^2. 
\end{align}  
Thanks to Poincar\'e inequalities \eqref{poincare_1} and \eqref{poincare_2}, by a similar argument as in \eqref{poincare_estimate_1}, we have
\begin{align*}
\od{}{t}\left(\norm{w }_{V _0}^2 + \norm{\nabla\xi}_{\LpP{2}}^2\right) &+\frac{\min\left\{\nu,\kappa\right\}}{8}\left(\norm{\nabla w_1}_{\LpP{2}}^2 +\frac{\norm{\nabla w_2}_{\LpP{2}}^2}{2} +\frac{\lambda_1}{2}\norm{w_2}_{\LpP{2}}^2 + \lambda_1\norm{\nabla \xi}_{\LpP{2}}^2\right) \notag\\
& + \frac{\mu}{2}\left(\norm{w_1}_{\LpP{2}}^2 + \norm{\nabla w_1}_{\LpP{2}}^2\right) \leq cK_3^2 \norm{\nabla\xi}_{\LpP{2}}^2, 
\end{align*}  
which implies that
\begin{align}\label{conv_2_2}
\od{}{t}\left(\norm{w }_{V _0}^2 + \norm{\nabla\xi}_{\LpP{2}}^2\right) &+\min\left\{\frac{\nu}{16},\frac{\nu\lambda_1}{16}, \frac{\kappa}{8}, \frac{\mu}{2}\right\}\left(\norm{w}_{V_0}^2 + \norm{\nabla \xi}_{\LpP{2}}^2\right)\notag \\ &\leq cK_3^2 \norm{\nabla\xi}_{\LpP{2}}^2. 
\end{align}  

Next, we observe that since $\nu \norm{\nabla w}_{\LpP{2}}^2 \geq 0$, under the assumption \eqref{mu_2} on $\mu$, equations \eqref{xi_ode_L2} and \eqref{ode_conv_1} imply that 
\begin{align}\label{3.32}
\od{}{t}\left(\norm{w}_{V_0}^2 + \norm{\xi}_{\LpP{2}}^2 \right) + \frac{\kappa}{8}\norm{\nabla \xi}_{\LpP{2}}^2 \leq 0. 
\end{align}
Integrating \eqref{3.32} over the interval $(t, t+\tau)$ and using estimate \eqref{asym_bound}, we conclude that
\begin{align}
\frac{\kappa}{8} \int_t^{t+\tau} \norm{\nabla \xi(s)}^2_{\LpP{2}} \,ds \rightarrow 0, \quad\text{as}\quad t\rightarrow \infty,
\end{align} 
for every $\tau>0$. Now we apply the general Gronwall Lemma \ref{gen_gron_2} to equation \eqref{conv_2_2}, while taking in Lemma \ref{gen_gron_2} $\alpha(t)$ $=$ $\min\left\{\frac{\nu}{16},\frac{\nu\lambda_1}{16}, \frac{\kappa}{8}, \frac{\mu}{2}\right\}$ and $\beta(t) = cK_3^2 \norm{\nabla\xi(t)}_{\LpP{2}}^2$, we conclude that
\begin{align}\label{exp}
\left(\norm{w(t) }_{V _0}^2 + \norm{\nabla\xi(t)}_{\LpP{2}}^2\right)\rightarrow 0,  
\end{align}
at an exponential rate, as $t\rightarrow \infty$. 
That is, 
$$\norm{u(t)-v(t)}_{V_0}^2 + \norm{\theta(t)-\eta(t)}_{	V_1}^2\rightarrow 0,$$ at an 
exponential rate, as $t \rightarrow \infty$. 

\end{proof}

\begin{theorem}\label{th_conv_2}
Suppose that $I_h$ satisfies the approximation property \eqref{app2} and the symmetry property \eqref{sym_1}. Let $(u(t),\theta(t))$, for $t\geq 0$, be a strong solution in the global attractor of \eqref{Bous_fun}. 
\begin{enumerate} 
\item Let $T>0$, $v_0 \in V_0$, and $\eta_0 \in H_1$.  Suppose that $\mu>0$ is large enough such that condition \eqref{mu_1} holds, and $h>0$ is small enough such that $2\mu c_0^2 h^2 \leq \frac{\nu}{16}$. Then, \eqref{DA_Bous_fun} has a unique solution $(v,\eta)$ that satisfies the regularity properties \eqref{weak}. 

Moreover, the solution $(v,\eta)$ depends continuously on the initial data in the $V_0\times{H_1}$ norm, and it satisfies $$\norm{u(t)-v(t)}_{V_0}^2 + \norm{\theta(t)-\eta(t)}_{H_1}^2\rightarrow 0,$$ at an 
exponential rate, as $t \rightarrow \infty$.

\item Let $T> 0$, $v_0\in V_0$, and $\eta_0\in V_1$. Suppose that $\mu>0$ is large enough such that condition \eqref{mu_2} holds, and suppose that $h>0$ is small enough such that $2\mu c_0^2 h^2 \leq \frac{\nu}{16}$. Then, \eqref{DA_Bous_fun} has a unique strong solution $(v,\eta)$ that satisfies the regularity properties \eqref{strong}. 

Moreover, the strong solution $(v,\eta)$ depends continuously on the initial data, in the $V_0\times{V_1}$ norm, and it satisfies 
$$\norm{u(t)-v(t)}_{V_0}^2 + \norm{\theta(t)-\eta(t)}_{V_1}^2\rightarrow 0,$$ at an 
exponential rate, as $t \rightarrow \infty$.  
\end{enumerate}
\end{theorem}

\begin{proof} 
The proof of this theorem is identical to the proof of Theorem \ref{th_conv_1} except for estimates \eqref{4} and \eqref{3b}. We will reproduce the adequate versions of these two estimates here. 

When the interpolant operator $I_h$ satisfies \eqref{app2}, using Young inequality, instead of the treatment in \eqref{4} we have  
\begin{align*}
&-\mu (I_h(w_1),w_1) = -\mu(I_h(w_1)-w_1, w_1) - \mu \norm{w_1}_{\LpP{2}}^2 \notag \\
& \qquad \qquad \leq \mu \norm{I_h(w_1)-w_1}_{\LpP{2}} \norm{w_1}_{\LpP{2}} - \mu\norm{w_1}_{\LpP{2}}^2\notag \\
&\qquad \qquad  \leq \mu c_0h\norm{w_1}_{H^1}\norm{w_1}_{\LpP{2}} + \mu c_0^2h^2\norm{w_1}_{H^2}\norm{w_1}_{\LpP{2}} - \mu\norm{w_1}_{\LpP{2}}^2 \notag \\
& \qquad \qquad \leq 2\mu c_0^2h^2\norm{w_1}_{H^1}^2 + \frac{\mu}{8}\norm{w_1}_{\LpP{2}}^2 + \frac{2\mu^2 c_0^4h^4}{\nu}\norm{w_1}_{H^2}^2 + \frac{\nu}{8}\norm{w_1}_{\LpP{2}}^2 - \mu\norm{w_1}_{\LpP{2}}^2. 
\end{align*}
Recall that, (see Remark \ref{poincare_velocity}), $\norm{w_1}_{H^2}^2 = \norm{w_1}_{H^1}^2 + 2\norm{\Delta w_1}_{\LpP{2}}^2$. Thus, thanks to the assumption $2\mu c_0^2 h^2 \leq \frac{\nu}{16}$, we get 
\begin{align*}
&-\mu (I_h(w_1),w_1)\notag \\ 
&\qquad \leq \left(2\mu c_0^2h^2 + \frac{2\mu^2 c_0^4 h^4}{\nu}\right)\norm{w_1}_{H^1}^2 +\frac{\nu}{8}\norm{w_1}_{\LpP{2}}^2 - \frac{7}{8}\mu \norm{w_1}_{\LpP{2}}^2 + \frac{2\mu^2 c_0^4h^4}{\nu}\norm{\Delta w_1}_{\LpP{2}}^2 \notag \\
& \qquad \leq \left(\frac{\nu}{8} +\frac{\nu}{128}\right) \norm{w_1}_{H^1}^2  + \frac{\nu}{8}\norm{w_1}_{\LpP{2}}^2 - \frac{7}{8}\mu \norm{w_1}_{\LpP{2}}^2 + \frac{\nu}{128}\norm{\Delta w_1}_{\LpP{2}}^2\notag \\
& \qquad \leq \frac{\nu}{4}\norm{w_1}_{H^1}^2 -\frac{7}{8}\mu\norm{w_1}_{\LpP{2}}^2 +\frac{\nu}{128}\norm{\Delta w_1}_{\LpP{2}}^2\notag \\
& \qquad = \frac{\nu}{4}\norm{\nabla w_1}_{\LpP{2}}^2 +\left(\frac{\nu}{4} -\frac{7}{8}\mu\right)\norm{w_1}_{\LpP{2}}^2 +\frac{\nu}{128}\norm{\Delta w_1}_{\LpP{2}}^2. 
\end{align*} 
Assumption \eqref{mu_1} implies that
\begin{align}\label{4_second}
-\mu (I_h(w_1),w_1) \leq \frac{\nu}{4}\norm{\nabla w_1}_{\LpP{2}}^2 -\frac{3}{4}\mu\norm{w_1}_{\LpP{2}}^2 +\frac{\nu}{128}\norm{\Delta w_1}_{\LpP{2}}^2. 
\end{align}

Instead of the treatment in \eqref{3b}, by using Young inequality, we have
\begin{align*}
&-\mu(I_h(w_1), -\Delta w_1) = \mu(I_h(w_1)-w_1), \Delta w_1) - \mu \norm{\nabla w_1}_{\LpP{2}}^2\notag \\
&\qquad \leq \mu c_0h\norm{w_1}_{H^1}\norm{\Delta w_1}_{\LpP{2}} + \mu c_0^2h^2\norm{w_1}_{H^2}\norm{\Delta w_1}_{\LpP{2}} - \mu \norm{\nabla w_1}_{\LpP{2}}^2 \notag \\
& \qquad \leq \frac{2\mu^2 c_0^2 h^2}{\nu}\norm{w_1}_{H^1}^2 + \frac{2\mu^2c_0^4h^4}{\nu}\norm{w_1}_{H^2}^2 + \frac{\nu}{4}\norm{\Delta w_1}_{\LpP{2}}^2 - \mu\norm{\nabla w_1}_{\LpP{2}}^2. 
\end{align*}
Since $\norm{w_1}_{H^2}^2 = \norm{w_1}_{H^1}^2 + 2\norm{\Delta w_1}_{\LpP{2}}^2$, the assumption $2\mu c_0^2 h^2 \leq \frac{\nu}{16}$ yields 
\begin{align*}
&-\mu(I_h(w_1), -\Delta w_1) \leq \left(\frac{\mu}{16} + \frac{\nu}{512}\right)\norm{w_1}_{H^1}^2 + \left(\frac{\nu}{512} + \frac{\nu}{4}\right)\norm{\Delta w_1}_{\LpP{2}}^2 - \mu\norm{\nabla w_1}_{\LpP{2}}^2\notag \\
& \qquad \leq \frac{3\nu}{8} \norm{\Delta w_1}_{\LpP{2}}^2 +  \left(\frac{\mu}{16} + \frac{\nu}{512}\right)\norm{w_1}_{\LpP{2}}^2 +  \left(\frac{\mu}{16} + \frac{\nu}{512}-\mu\right)\norm{\nabla w_1}_{\LpP{2}}^2. 
\end{align*}
Thus, condition \eqref{mu_1} implies that 
\begin{align}\label{3b_second} 
&-\mu(I_h(w_1), -\Delta w_1) \leq \frac{3\nu}{8} \norm{\Delta w_1}_{\LpP{2}}^2 +\frac{\mu}{8}\norm{w_1}_{\LpP{2}}^2 - \frac{7\mu}{8}\norm{\nabla w_1}_{\LpP{2}}^2. 
\end{align}

The rest of the proof of the theorem follows the proof of Theorem \ref{th_conv_1} while replacing \eqref{4} and \eqref{3b} by \eqref{4_second} and \eqref{3b_second}, respectively.  

\end{proof} 

\bigskip
\section*{Acknowledgements}

The work of A.F. is supported in part by NSF grant  DMS-1418911. The work of E.L. is supported  by the ONR grant N0001415WX01725.  The work of  E.S.T.  is supported in part by the ONR grant N00014-15-1-2333 and the NSF grants DMS-1109640 and DMS-1109645.  

\bigskip 

\end{document}